\documentclass[a4paper,10pt]{article}
\usepackage{LatexDefinitions}
\usepackage{amssymb}
\usepackage{caption}
\usepackage{dsfont}
\usepackage{dsfont}
\usetikzlibrary{patterns}
\newtheorem*{theorem*}{Theorem}

\newcommand{\E}{\mathbb{E}}

\newcommand{\Pb}{\mathbb{P}}
\newcommand{\pl}[0]{f_M^N}
\newcommand{\epl}[0]{f_M}
\newcommand{\lpl}[0]{f_\infty}
\newcommand{\mms}[0]{\theta_M^N}
\newcommand{\emms}[0]{\theta_M}
\newcommand{\lmms}[0]{\theta_\infty}
\newcommand{\lip}{\mathcal{L}}
\newcommand{\bk}[1]{\left(#1\right)}

\definecolor{ab}{rgb}{0.94, 0.97, 1.0}
\title{Density estimation from batched broken random samples}
\author{Hancheng Bi, Bernhard Schmitzer, Thilo D. Stier}

\begin{document}
\maketitle
\begin{abstract}
The \textit{broken random sample} problem was first introduced by DeGroot, Feder, and Gole (1971, Ann.\ Math.\ Statist.): in each observation (batch), a random sample of $M$ i.i.d.\ point pairs $ ((X_i,Y_i))_{i=1}^M$ is drawn from a joint distribution with density $p(x,y)$, but we can observe only the unordered multisets $(X_i)_{i=1}^M$ and $(Y_i)_{i=1}^M$ separately; that is, the pairing information is lost. For large $M$, inferring $p$ from a single observation has been shown to be essentially impossible. In this paper, we propose a parametric method based on a pseudo-log-likelihood to estimate $p$ from $N$ i.i.d.\ broken sample batches, and we prove a fast convergence rate in $N$ for our estimator that is uniform in $M$, under mild assumptions.
\end{abstract}

\section{Introduction}
\label{sec:Intro}
\subsection{Problem statement and related work}
\label{sec:IntroProbStatement}
\paragraph{Problem statement.} Let $\mathbb{X}, \mathbb{Y}$ be measurable spaces and $\pi\in \mathcal{P}(\mathbb{X}\times \mathbb{Y})$ be a joint probability measure. 
Let $(X_i, Y_i)_{i=1}^M$ be independent and identically distributed (i.i.d.) pairs of random variables with law $\pi$. We assume that we do not have access to the actual pairing information, i.e.~we will only observe the random \emph{multisets} $\{X_1,\cdots,X_M\}$ and $\{Y_1,\cdots,Y_M\}$ but not the ordering of the points. 
This is called a \textit{broken random sample} from $\pi$, originally introduced in \cite{degroot1971matchmaking}. Previous research on extracting information from a broken sample when $\pi$ is a bi-variate normal distribution includes~\cite{degroot1980} and \cite{chan2001file}. In particular, \cite{bai2005broken} proved that estimating the correlation parameter is not possible under some mild conditions. In this paper we assume that instead of observing a single broken sample, we have access to $N$ i.i.d.~observations of broken samples, i.e.~we have i.i.d.~pairs of random variables $((X_i^k,Y_i^k)_{i=1}^M)_{k=1}^N$ with law $\pi$ and we observe the collection of multisets
\begin{align*}
    \Big(\{X_1^k,\cdots,X_M^k\},\{Y_1^k,\cdots,Y_M^k\}\Big)_{k=1}^N.
\end{align*}

We are interested in estimating $\pi$ in a parametric setting from these i.i.d.~samples and investigate the behaviour of the estimator with respect to $N$ and $M$.
We refer to each observed pair of multisets $(\{X_1^k,\cdots,X_M^k\},\{Y_1^k,\cdots,Y_M^k\})$ as a \emph{batch} and to the problem of estimating  $\pi$ as the \emph{batched broken sample problem}.

If $M = 1$ we recover the classical problem of estimating $\pi$ from $N$ i.i.d.~samples. 
However, due to the missing pairing information, as $M \to \infty$, intuitively the batches will be approximately independent samples from the marginal distributions of $\pi$, which leads to the suspicion that the information about $\pi$ carried by each batch could decrease to zero as $M$ increases. 
In the following we prove that this is not the case.

\paragraph{Applications.}
The batched broken sample problem serves as a idealised model for particle colocalisation analysis. \Cref{fig:coloc} shows part of a super-resolution fluorescence microscopy image of immunolabelled human cells. The HA-tag in the mitochondrial outer membrane and the protein Mic60 in the mitochondrial inner membrane are tagged with green and purple markers, respectively. Most proteins appear to be located within close spatial proximity of a protein from the other species. Is this because the proteins form bound pairs to realize their biological function, or are the protein locations independent from each other and colocalisation is merely accidental, stemming from the high density of particles in a confined region? If one models the protein locations of the two species as random variables $X_i$ and $Y_i$, as above, then the two situations can be distinguished, based on the law of $\pi$: For bound pairs, $X_i$ and $Y_i$ will be close with high probability; in the case of accidental colocalisation, $\pi$ is an independent product measure.
Of course, to realistically estimate $\pi$ from images one must take into account several additional effects: not all particles are actually captured in the imaging process; even when particles do tend to form pairs, not all particles are necessarily paired at all times; and particles in different images will follow slightly different laws $\pi$. However, our simplified model does demonstrate that increasing particle density is not necessarily a problem for inferring $\pi$.

\begin{figure}[]
    \centering
    \includegraphics[width = \textwidth]{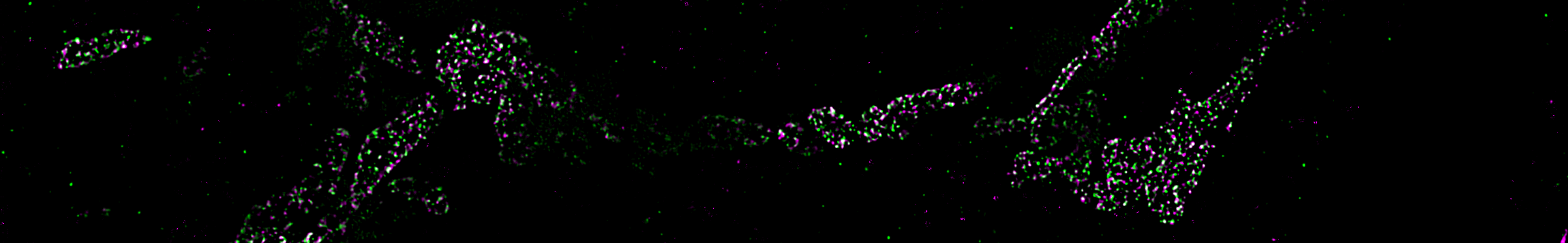}
    \caption{Stimulated emission depletion (STED) microscopy image, part of \cite[Fig 37(A)]{dohrke2024puck}. Cells were stained for
    the HA-tag (green) and Mic60 (purple).}
    \label{fig:coloc}
\end{figure}

Another application is the analysis of time-discrete dynamical systems. In this case we interpret $X_i$ and $Y_i$ as positions of a particle at two subsequent time steps and a common task is to infer the conditional distribution of $Y_i$, given $X_i$. $M=1$ corresponds to observing the evolution of a single particle whereas $M>1$ corresponds to the case where we observe multiple indistinguishable particles, such as floaters in a fluid as in particle image velocimetry, where small particles are seeded into the fluid and illuminated by a pulsed laser. The distribution of light scattered by these particles is recorded each time the laser pulses, thereby creating a discrete sequence of images (see, e.g., \cite[Fig. 6]{AguehPIV2015}). The task is to retrieve the flow velocity. 

When the point density is relatively low, or when points move slowly, both of the above problems can be tackled using optimal-transport-based methods; see \cite{tameling2021colocalization} and \cite{AguehPIV2015,OTCoherentSet2021}, respectively. However, optimal transport tends to bias the pairing towards coincidentally close points, which becomes more likely as the point density increases. In this paper, we introduce a parametric method to estimate the density of the law $\pi$, and we prove quantitatively that the method is not affected in the high-density regime (i.e.~as $M \to \infty$). A non-parametric method using a similar loss function has been discussed in \cite{TransferOp}, along with a qualitative proof of convergence as $N \to \infty$ in a suitable sense.

\subsection{Outline}
Throughout the rest of this section we introduce assumptions and notation, state the problem in detail and preview the main results. 
In \Cref{sec:theory} we collect or prove necessary auxiliary theorems and lemmata. Our main results are proven in \Cref{sec:proof}.
Some numerical illustrations are shown in \Cref{sec:num}.

\subsection{Problem description and main results}
Let $\mu \in \mathcal{P}(\mathbb{X})$ and $\nu \in \mathcal{P}(\mathbb{Y})$ be the marginal distributions of $\pi$. 
In a setting where $\pi$ can be identified from the marginals, the problem reduces to conventional density estimation.
Therefore, throughout this paper, we always assume $\mu, \nu$ are known, and we are interested in estimating the density of $\pi$ with respect to $\mu\otimes\nu$, analogous to \cite[p.~529]{bai2005broken}.
In practice one may first estimate $\mu$ and $\nu$ with very high precision from the $M \cdot N$ i.i.d.~random variables $(X_i^k)_{i,k}$ and $(Y_i^k)_{i,k}$.

\begin{assumption}
    $\pi = p^* \cdot (\mu\otimes\nu)$ for some $p^* : \mathbb{X}\times \mathbb{Y} \to \mathbb{R}$.
\end{assumption}
\noindent We also need to define a candidate set from which we pick the estimator. Here we use a parametric class.
\begin{definition}\label{def:candidate}
Let $\Theta$ be an open subset of finite dimensional Euclidean space with compact closure $\Bar{\Theta}$, let $D:= \sup_{\theta,\theta'\in\Theta}\norm{\theta - \theta'}$ be its diameter.
Let $(p^\theta)_{\theta\in\Bar\Theta}$ be a parametric class of densities with the following properties:
\begin{enumerate}
    \item The set of functions $\{\Bar{\Theta} \ni \theta\mapsto p^\theta(x,y) \mid (x,y)\in\mathbb{X}\times\mathbb{Y}\}$ is twice continuously differentiable when restricted to $\Theta$, equibounded by some constant $0 < U < \infty$, equi-Lipschitz with some Lipschitz constant $\mathcal{L} < \infty$, and all directional derivatives of second order with respect to $\theta\in\Theta$ are equibounded by $\lip' < \infty$.
    \item For all $\theta\in\Theta$, $x \in \mathbb{X}$, $y \in \mathbb{Y}$, we have $\int_\mathbb{X} p^\theta(x', y) \diff\mu(x') = \int_\mathbb{Y} p^\theta(x, y') \diff\nu(y') = 1$. 
\end{enumerate}
\end{definition}

\begin{remark}[Breaking the samples]\label{rem:fl_BreakingSamples}
    In order to `forget' the pairing information within batches, the authors of \cite{TransferOp} introduced an auxiliary random variable, a uniformly sampled permutation $\sigma \sim \mathcal{U}(\text{Sym}(M))$, and assume that, for each batch one can only observe $((X_i,Y_{\sigma(i)}))_{i=1}^M$. As discussed in \cite{TransferOp}, in this model the negative log-likelihood is given by
    \begin{align}\label{eq:loglikelihood}
    \ell(\theta) = -\frac{1}{N}\sum_{k=1}^N \log\left(\frac{1}{M!}\sum_{\sigma\in S_M} \prod_{i=1}^M p^\theta\left(X_i^k, Y_{\sigma(i)}^k\right)\right).
    \end{align}
    While this is a rigorous way to formulate the broken-sample problem, it introduces an extra random variable and complicates the probability space. 

    In this paper, we take a different approach for the sake of simplicity:
    Instead of explicitly modelling the broken samples (e.g.~by applying an unknown uniformly random permutation to the points), we model the data with a known pairing but make sure that our estimator does not use any pairing information. Explicitly this emerges as the estimator being defined as the minimiser of a loss function which is invariant under within-batch permutation of the points.
\end{remark}

Clearly, \Cref{eq:loglikelihood} is permutation invariant and therefore in principle it can be used as loss function. However it is not feasible in practice unless $M$ is very small, since the complexity of evaluation of $\ell(\theta)$ is in general NP-hard \cite{VALIANT1979189}. Inspired by \cite{TransferOp}, we define our loss function through a negative pseudo log-likelihood
\begin{align}
\label{eq:pl}
    \pl(\theta) \assign -\frac{1}{N}\sum_{k=1}^N \sum_{i=1}^M\sum_{j=1}^M \log\left(\frac{1}{M}p^\theta(X_i^k,Y_j^k) + \frac{M-1}{M}\right) \quad \text{and its minimiser} \quad \mms \in \argmin_{\theta\in\Bar\Theta} \pl(\theta).
\end{align}
In addition, we introduce
\begin{align*}
    \epl(\theta) & := \mathbb{E}(\pl(\theta)), &
    \lpl(\theta) & := \lim_{M\to\infty} \epl(\theta), &
    \emms & \in \argmin_{\theta\in\Bar\Theta} \epl(\theta), &
    \lmms & \in \argmin_{\theta\in\Bar\Theta} \lpl(\theta).
\end{align*}
At first glance, one might suspect that $\pl$ scales as $O(M^2)$ as $M \to \infty$ because of the double sum over $i$ and $j$. This concern is resolved by \Cref{thm:mean_pseudo}, which shows that the expectation $\epl(\theta)$ has a finite limit as $M \to \infty$.

\begin{remark}\label{rem:intuition}
    Intuitively the negative pseudo-log-likelihood \eqref{eq:pl} arises as follows.
    When uniformly choosing a random pair $(i, j)$ of point indices in some batch $k$, there is a $\frac{1}{M}$ chance that $X := X^k_i$ and $Y := Y^k_j$ form a pair with distribution $\pi$ and a $\frac{M - 1}{M}$ chance that they are independent, i.e.~$(X,Y) \sim \left(\frac{1}{M} p^* + \frac{M-1}{M}\right) (\mu\otimes\nu)$.
    Different pairs $(X^k_i,Y^k_j)$ (for the same $k$, but different $i,j$) will in general not be independent from each other. However, if we allow for this approximation, \eqref{eq:pl} arises as the sum of the (negative) log-likelihoods of all pairs.
\end{remark}

\begin{remark}[Comparison with \cite{TransferOp}]
In \cite{TransferOp} the following pseudo-likelihood
\begin{align}
    \label{eq:plMix}
    \theta \mapsto -\frac{1}{N}\sum_{k=1}^N \sum_{j=1}^M\log\left(\frac{1}{M}\sum_{i=1}^M p^\theta(X_i^k,Y_j^k) \right)
\end{align}
is considered instead of \eqref{eq:pl}. This is also permutation invariant. The term $\frac{1}{M}\sum_{i=1}^M p^\theta(X_i^k,Y_j^k)$ appearing within the logarithm corresponds to the conditional probability of $Y_j^k$, conditioned on the tuple $(X_i^k)_{i=1}^M$, if one assumes that the pairing information is unavailable. \eqref{eq:plMix} then arises by arguing analogously to \Cref{rem:intuition} by making the approximation that the variables $(Y_j^k)_{j=1}^M$ are independent, when conditioned on $(X_i^k)_{i=1}^M$.
Using techniques from \cite{TransferOp} it can be argued that \eqref{eq:plMix} is a more accurate approximation than \eqref{eq:pl}, however we consider the latter in this manuscript as it is more amenable to quantitative convergence analysis.
\end{remark}

\noindent We can now state our main results.
The first result is on the asymptotic form of the expected pseudo likelihood as $M \to \infty$.
\begin{theorem}
    \label{thm:mean_pseudo}
    \begin{align*}
        \epl(\theta) &=  \lpl(\theta) + O(M^{-1}) \\
        \intertext{where}
        \lpl(\theta) &= \frac{1}{2}\Big(\norm{p^\theta- p^*}^2_{L^2(\mu\otimes\nu)} - \norm{p^*}^2_{L^2(\mu\otimes\nu)} + 1\Big).
    \end{align*}
\end{theorem}

\begin{remark}[Behaviour of minimizers of $\epl$]\label{rem:KL}
    When $M=1$, $\epl$ reduces to the standard population negative log-likelihood. It is well known that minimising the negative log-likelihood is equivalent to minimising the $\KL$ divergence between the true distribution and our estimated one. For arbitrary $M$, analogously to \cite[Proposition A.4]{TransferOp}, we have
    \begin{align*}
        \epl (\theta) 
        &=
        M^2 \cdot \KL\left((\tfrac{1}{M}p^* + \tfrac{M-1}{M}) \cdot \mu\otimes\nu \ \middle|
        \ (\tfrac{1}{M}p^\theta + \tfrac{M-1}{M}) \cdot \mu\otimes\nu \right) + \const.
    \end{align*}
    This follows from \Cref{eq:Eql_integral} in the proof of \Cref{thm:mean_pseudo}. Hence, if $p^*\in(p^\theta)_{\theta\in\Theta}$, we conclude that any minimiser of $\epl$ equals $p^*$ $\mu\otimes\nu$-almost everywhere.
    If $p^*\notin(p^\theta)_{\theta\in\Theta}$, then for $M=1$, minimizers of $\epl$ will be KL-projections of $p^*$ onto the parametric class.
    By \Cref{thm:mean_pseudo}, as $M \to \infty$ minimizers of $\epl$ should converge to the $L^2(\mu\otimes \nu)$-projection.
    This becomes more concrete in \Cref{thm:main-estimator}.
\end{remark}

The next result is on the concentration of $\pl$ around $\epl$, uniformly in $\theta$.
\begin{theorem}
\label{thm:main_detailed}
    Let $D,U,\mathcal{L}$ be as in \Cref{def:candidate}, let
    \begin{align*}
        J := \int_0^D \log\left(1+ N(\varepsilon;\Theta, \norm{\cdot})\right) \diff\varepsilon < \infty,
    \end{align*}
    in which $N(\varepsilon;\Theta, \norm{\cdot})$ is the covering number of $(\Theta, \norm{\cdot})$ with radius $\varepsilon$. Note that $J$ is finite under our assumption for $\Theta$. Then there exists an absolute constant $c>0$ such that
    \begin{align*}
        \Pb\left(\sup_{\theta\in\Theta}\abs{\pl(\theta) - \epl(\theta)} \geq 6t\right) 
        \leq 
        6\exp\left(-c\min\Big\{ \frac{t^2 N}{U^4}, \frac{tN}{U}\Big\}\right) 
        + 6\exp\left(\frac{J}{D}\right)\exp\left(-\frac{ct\sqrt{N}}{D\mathcal{L}U}\right)
    \end{align*}
    for all $t > 0$.
\end{theorem}

\noindent This suggests that the estimation quality will benefit from large $N$ as expected and will not be harmed by large $M$. 
Instead, in combination with \Cref{thm:mean_pseudo}, for large $M$ and $N$ we see that one can expect the estimator of $p^*$ to be close to the $L^2(\mu\otimes\nu)$-projection of $p^*$ onto $(p^\theta)_{\theta \in \Bar\Theta}$. As a consequence, when assuming uniqueness of this projection and strictly positive curvature of $\|p^\theta- p^*\|^2_{L^2(\mu\otimes\nu)}$ around it, one can get a lower bound of the convergence rate of the minimiser of $\pl(\theta)$ to the minimiser of $\epl(\theta)$ for large $M$:
\begin{theorem}
\label{thm:main-estimator}
    Assume that
    $\lmms\in\Theta$ is the unique minimiser of $\lpl$, and there exist some $\tau, r > 0$ such that the smallest eigenvalue of $\nabla^2 \lpl(\theta)$ is larger than $\tau$ for all $\theta\in\Theta$ with $\norm{\theta - \lmms}_2 \leq r$.
    Then there exist some $M_0\in \N, t_0 > 0$ s.t.~for all $M \geq M_0$ and $t \leq t_0$ we have
    \begin{align*}
        \Pb\left(\norm{\mms - \emms}_2 \geq t\right) \leq C_1\exp\left(-C_2\min\left\{t^2 N, t\sqrt{N}\right\}\right),
    \end{align*}
    where the empirical minimizer $\theta_M$ is unique and $C_1,C_2 >0$ are constants independent of $M$ and $N$.
\end{theorem}

\section{Preliminaries}\label{sec:theory}

\begin{theorem}[Differentiability of parametrised integrals \protect{\cite[Theorem 3.18]{amann2009analysis}}]\label{thm:interchange}
Suppose $\Theta$ is open in $\mathbb{R}^n$, and suppose $f : \mathbb{X}\times \Theta \to \mathbb{R}$ satisfies
\begin{enumerate}
    \item $f(\cdot, \theta) \in L^1_{\mu}(\mathbb{X})$ for every $\theta\in\Theta$;
    \item $f(x,\cdot) \in C^1(\Theta)$ for $\mu$-almost every $x\in\mathbb{X}$;
    \item there exists some $g\in L^1_\mu(\mathbb{X})$ such that
        \begin{align*}
            \Big|\frac{\partial}{\partial \theta_m} f(x,\theta)\Big| \leq g(x) \qquad \text{for all }(x,\theta)\in \mathbb{X}\times\Theta \quad \text{and} \quad m \in \{1,\ldots,n\}.
        \end{align*}
\end{enumerate}
Then $\theta \mapsto \int_\mathbb{X} f(x,\theta) \diff \mu$ is continuously differentiable and
\begin{align*}
    \frac{\partial}{\partial \theta_m}\int_\mathbb{X} f(x,\theta) \diff \mu(x) = \int_\mathbb{X} \frac{\partial}{\partial \theta_m} f(x,\theta) \diff \mu(x) \qquad \text{for all } \theta \in \Theta \quad \text{and} \quad m \in \{1,\ldots,n\}.
\end{align*}
\end{theorem}

\begin{definition}[$\psi_1$-norm]\label{def:psi_norm}\cite[eqn. (2.21)]{vershynin2018high}
    The sub-exponential norm of a real valued random variable $F$ is defined by
    \begin{align*}
        \norm{F}_{\psi_1} = \inf \left\{t > 0 : \E \exp\Big(\frac{\abs{F}}{t}\Big) \leq 2\right\}.
    \end{align*}
\end{definition}

\begin{lemma}\label{lem:psi_norm}
    Given a random variable $F$ with $\E(\abs{F}^p) \leq (Cp)^p$ for all $p\in\N$, we have $\norm{F}_{\psi_1} \leq 2eC$.
\end{lemma}
\begin{proof}
Substitute $t = 2eC$ into the following equation
    \begin{align*}
        \E \exp\Big(\frac{\abs{F}}{t}\Big)
        &=
        \E \left(\sum_{p = 0}^\infty \frac{\abs{F}^p}{t^p p!}\right)
        \leq
        \sum_{p=0}^\infty \frac{C^p p^p}{t^p p!}
        \leq 
        \sum_{p=0}^\infty \frac{C^p e^p}{t^p}
        =
        \frac{1}{1-\frac{Ce}{t}}
    \end{align*}
    where we used $(p/e)^p \leq p!$ as well as the limit of the geometric series.
\end{proof}

\begin{theorem}[Bernstein, Corollary 2.8.3 in \cite{vershynin2018high}]\label{thm:Bernstein}
    Let $F_1,..., F_N$ be i.i.d.~sub-exponential random variables with mean zero, then for any $t\geq 0$  we have
    \begin{align*}
        \Pb\left(\abs{\frac{1}{N}\sum_{i = 1}^N F_i} \geq t\right) 
        \leq
        2\exp\left(
        -c\min\left\{
            \frac{t^2 N}{\norm{F_1}_{\psi_1}^2}, 
            \frac{tN}{\norm{F_1}_{\psi_1}}
        \right\}
        \right)
    \end{align*}
    in which $c$ is an absolute constant.
\end{theorem}

\begin{lemma}\label{lem:psi_norm_iid}
    Let $F_1,..., F_N$ be i.i.d.~sub-exponential random variables with mean zero, then for some absolute constant $c > 0$ we have
    \begin{align*}
        \norm{\frac{1}{N}\sum_{i = 1}^N F_i}_{\psi_1} \leq \frac{c \norm{F_1}_{\psi_1}}{\sqrt{N}}.
    \end{align*}
\end{lemma}
\begin{proof}
    Denote $\beta := \norm{F_1}_{\psi_1}$. For any $p \geq 1$, by \Cref{thm:Bernstein} 
    \begin{align*}
        \E\left(\abs{\frac{1}{N}\sum_{i=1}^N F_i}^p\right)
        &= 
        \int_{0}^\infty \Pb\left(\abs{\frac{1}{N}\sum_{i = 1}^N F_i}^p \geq t\right) \diff t
        \\ &=
        p\int_{0}^\infty \Pb\left(\abs{\frac{1}{N}\sum_{i = 1}^N F_i} \geq t\right) t^{p-1} \diff t
        \\ &\leq
        2p\int_{0}^\infty t^{p-1} \exp\Big(-\frac{cNt^2}{\beta^2}\Big) \diff t
        +
        2p\int_{0}^\infty t^{p-1} \exp\Big(-\frac{cNt}{\beta}\Big) \diff t
        \\ & \leq
        2p \Big(\frac{cN}{\beta^2}\Big)^{-p/2} \Gamma\Big(\frac{p}{2}\Big) + 2p \Big(\frac{cN}{\beta}\Big)^{-p}\Gamma(p)
        \\ & \leq
        6p\Gamma(p)\Big(\frac{\beta}{c \sqrt{N}}\Big)^p \leq \Big(\frac{6\beta p}{c \sqrt{N}}\Big)^p.
    \end{align*}
    In the last line we used that the $c$ from \Cref{thm:Bernstein} w.l.o.g.~fulfills $c \leq 1$ and that for $p \geq 1$ we have $\Gamma(p/2) \leq 2 \Gamma(p)$ as well as $p \Gamma(p) = p! \leq p^p$.
    Applying \Cref{lem:psi_norm} finishes the proof.
\end{proof}

\begin{definition}[$\psi_1$-process, Definition 5.35 in \cite{wainwright2019high}]\label{def:psi-process} A zero mean stochastic process $\{F_\theta: \theta \in \Theta\}$ is a $\psi_1$ process with respect to a norm $\norm{\cdot}_\Theta$ if
\begin{align*}
    \norm{F_\theta - F_{\theta'}}_{\psi_1} \leq \norm{\theta - \theta'}_\Theta.
\end{align*}
\end{definition}

\begin{lemma}\label{lem:trivial}
    Consider a parametric class of functions $g_\theta : \mathbb{X}\to \R$, a probability measure $\rho \in \mathcal{P}(\mathbb{X})$, and some i.i.d.~$X_k\sim \rho$. 
    If there exists an $L>0$ such that for any $\theta,\theta'\in \Theta$ one has $\norm{g^\theta - g^{\theta'}}_\infty \leq L\norm{\theta - \theta'}_\Theta$, then
    \begin{align*}
        F_N(\theta) := \frac{1}{N}\sum_{k=1}^N g^\theta(X_k) - \E g^\theta(X)
    \end{align*}
    is a $\psi_1$-process with respect to the norm $\frac{cL}{\sqrt{N}}\norm{\cdot}_\Theta$ for an absolute constant $c > 0$.
\end{lemma}
\begin{proof}
    Note that $\E (\abs{F_1(\theta) - F_1(\theta')}^p) \leq (2L\norm{\theta - \theta'})^p$, apply \Cref{lem:psi_norm} and \Cref{lem:psi_norm_iid}. 
\end{proof}

\begin{definition}[Generalized Dudley entropy integral]\label{def:Dudley} Take $D$ from \Cref{def:candidate}. We define
\begin{align*}
    \mathcal{J}(\delta, D) = \int_\delta^D \log\Big(1 + N(\varepsilon; \Theta, \norm{\cdot}_\Theta)\Big) \diff\varepsilon
\end{align*}
where $N(\varepsilon; \Theta, \norm{\cdot}_\Theta)$ is the $\varepsilon$-covering number for $(\Theta, \norm{\cdot}_\Theta)$.
\end{definition}

\begin{theorem}[Theorem 5.36 in \cite{wainwright2019high}]\label{thm:Dudley}
    Let $\{F_\theta : \theta \in \Theta\}$ be a $\psi_1$-process with respect to $\norm{\cdot}_\Theta$, then there is a universal constant $c_1$ such that for all $t\in \R$ we have
    \begin{align*}
        \Pb\left(\sup_{\theta,\theta' \in \Theta} \abs{F_\theta - F_{\theta'}} \geq c_1 (\mathcal{J}(0, D) + t)\right) \leq 2\exp\Big(-\frac{t}{D}\Big).
    \end{align*}
\end{theorem}

\begin{lemma}\label{lem:Dudley_iid}
    Take $D$ from \Cref{def:candidate}, and let $\norm{\cdot}$ denote the standard Euclidean norm. Define
    \begin{align*}
        J := \int_0^D \log\left(1+ N(\varepsilon;\Theta, \norm{\cdot})\right) \diff\varepsilon.
    \end{align*}
    Let $f(\theta)$ be a $\psi_1$-process w.r.t.~$\frac{C}{\sqrt{N}}\norm{\cdot}$ for some $C>0$.
    There is an absolute constant $c_1>0$ such that for all $t > 0$
    \begin{align*}
        \Pb\left(\sup_{\theta,\theta'\in\Theta}\abs{f(\theta) - f(\theta')} \geq t\right) \leq 2\exp\left(\frac{J}{D}\right)\exp\left(-\frac{t\sqrt{N}}{c_1DC}\right).
    \end{align*}
\end{lemma}
\begin{proof}
    First note that since $\Bar{\Theta}
    $ is a compact subspace of finite dimensional Euclidean space, we have $J < \infty$ by \cite[Proposition 5]{cucker2002mathematical}.
    Let 
    $D':= \sup_{\theta,\theta'\in\Theta}\frac{C}{\sqrt{N}}\norm{\theta-\theta'} = \frac{DC}{\sqrt{N}}$.
    From \Cref{def:Dudley}, using a change of variables in the integral we get
    \begin{align}\label{eqn:iidJ}
        \mathcal{J}(0,D') 
        &= 
        \int_0^{D'} 
        \log\left(1+N\Big(\varepsilon;\Theta, \frac{C}{\sqrt{N}}\norm{\cdot}\Big)\right) \diff \varepsilon 
        = 
        \frac{CJ}{\sqrt{N}}.
    \end{align}
    By \Cref{thm:Dudley} we have for some absolute constant $c_1>0$
    \begin{align*}
        \Pb\left(\sup_{\theta,\theta'\in\Theta}\abs{f(\theta) - f(\theta')} \geq c_1\Big(\frac{CJ}{\sqrt{N}} + \tau\Big)\right) 
        &\leq 
        2\exp\Big(-\frac{\tau}{D'}\Big) = 2\exp\Big(-\frac{\tau\sqrt{N}}{DC}\Big) 
        \qquad 
        \forall \tau\in\R.
    \end{align*}
    Substituting $t = c_1\Big(\frac{CJ}{\sqrt{N} } + \tau\Big)$ finishes the proof.
\end{proof}

\begin{lemma}\label{lem:aux}
    Let $\{F_\theta : \theta \in \Theta\}$ be a stochastic process and fix some $\theta_0 \in \Theta$. Then
    \begin{align*}
        \Pb\left(\sup_{\theta\in\Theta}\abs{F_\theta} \geq 2t\right) 
        &\leq 
        \Pb\left(\sup_{\theta, \theta'\in\Theta} \abs{F_\theta - F_{\theta'}} \geq t\right) 
        + \Pb\left(\abs{F_{\theta_0}} \geq t\right).
    \end{align*}
\end{lemma}

\begin{proof}
    Apply triangle inequality and union bound.
\end{proof}

\begin{lemma}[Hoeffding, Theorem 2.2.6 in \cite{vershynin2018high}]\label{lem:Hoeffding}
    Let $F_1, ..., F_N$ be i.i.d., assume that $\abs{F_k}\leq C$, then for any $t\geq 0$, we have
    \begin{align*}
        \Pb\left(\abs{\frac{1}{N}\sum_{k=1}^N F_k - \E(F_1)} \geq t \right) \leq 2\exp\left(-\frac{Nt^2}{2C^2}\right).
    \end{align*}
\end{lemma}

\section{Non-asymptotic convergence results}\label{sec:proof}
\subsection[Concentration of f\_M\textasciicircum{}N(theta)]{Concentration of $\pl(\theta)$}\label{sec:concentration}
\noindent 
From now on we assume $M\geq 2$. In this section we will abbreviate the random variables $p^\theta(X_i^k,Y_j^k)$ as $p_{ij}^k(\theta)$. Similarly, for points $(x_i)_i \subset \mathbb{X}$, $(y_i)_i \subset \mathbb{Y}$, we abbreviate $p^*(x_i,y_i)$ as $p^*_{ii}$ and $p^\theta(x_i,y_j)$ as $p_{ij}(\theta)$. We also define $q_{ij}^k(\theta) := p_{ij}^k(\theta) - 1$ and $q_{ij}(\theta) = p_{ij}(\theta) - 1$. 
Note that $p^{\ast}_{ii}$ and $p_{ij}$ integrate to $1$ on both marginals, i.e.~we have
\begin{align*}
    \int_{\mathbb{X}} p^{\ast}_{ii} \diff \mu(x_i)
    = 
    \int_{\mathbb{X}} p^{\ast}_{ii} \diff \nu(y_i)
    =
    1
\end{align*}
and a corresponding statement for $p_{ij}(\theta)$. Similarly $q_{ij}(\theta)$ integrates to $0$.
We write $\diff \pi_{ii}$ for integrating against $\diff \pi(x_i,y_i)$ and $\diff \pi^{\otimes M}$ for the joint integration against all factors $\Pi_{i=1}^M \diff \pi_{ii}$. We use similar notation for integrals against the product measure $\mu \otimes \nu$.

We now prove a statement about the asymptotic form of our expected pseudo-likelihood, as $M \to \infty$.
\begin{theorem*}[Recollection of \protect{\Cref{thm:mean_pseudo}}]
    $\epl(\theta) = \frac{1}{2}\Big(\norm{p^\theta- p^*}^2_{L^2(\mu\otimes\nu)} - \norm{p^*}^2_{L^2(\mu\otimes\nu)} + 1\Big) + O(M^{-1})$.
\end{theorem*}
\begin{proof}
    Recall $q_{ij}(\theta) := p_{ij}(\theta) - 1$. We get
    \begin{align*}
        -\epl (\theta)
        &=
        \int \sum_{i,j=1}^M\log\Big(\frac{1}{M}p_{ij}(\theta) + \frac{M-1}{M}\Big) \diff \pi^{\otimes M}
        \\&=
        \int \sum_{i,j=1}^M \log\Big(1 + \frac{1}{M}q_{ij}(\theta)\Big)\ \prod_{l=1}^M p^*_{ll} \ \diff(\mu\otimes\nu)^{\otimes M}
        \\&=
        \sum_{i=1}^M \int \log\Big(1 + \frac{1}{M}q_{ii}(\theta)\Big)\ \prod_{l=1}^M p^*_{ll} \ \diff(\mu\otimes\nu)^{\otimes M} \\
        &\qquad + 
        \sum_{i\neq j} \int \log\Big(1 + \frac{1}{M}q_{ij}(\theta)\Big)\ \prod_{l=1}^M p^*_{ll} \ \diff(\mu\otimes\nu)^{\otimes M}.
    \end{align*}
        For the first part, consider the summand $i=1$. It equals $\int \log (1+q_{11}/M) p^*_{11}\prod_{l\neq 1} p_{ll}^* \diff(\mu\otimes\nu)^{\otimes M},$ and one can integrate over $(x_l,y_l)$ for all $l\neq 1$ by using $\int p^*_{ll} d\mu\otimes\nu = 1$. One can argue likewise for all other values of $i$, yielding $M$ times the same value.
        For the second part, similarly consider the summand $i=1$ and $j=2$. It equals $\int \log(1+q_{12}/M)p^*_{11}p^*_{22}\prod_{l=3}^M p^*_{ll} \diff(\mu\otimes\nu)^{\otimes M}$, we can again integrate over $(x_l,y_l)$ for $l\geq 3$, and then integrate over $y_1$ and $x_2$ by $\int_\mathbb{X} p^*_{22} \diff\mu(x_2) = \int_\mathbb{Y} p^*_{11} \diff\nu(y_1) = 1$. Arguing likewise for all other values of $i,j$, one obtains $M(M-1)$ times this contribution. Therefore,
    \begin{align}
        -\epl (\theta) &=
        M\int \log\Big(1+\frac{1}{M}q_{11}(\theta)\Big) p^*_{11} \diff\mu\otimes\nu + M(M-1)\int \log\Big(1+\frac{1}{M}q_{11}(\theta)\Big) \diff\mu\otimes\nu
        \nonumber \\&=
        M^2 \int \log\Big(1+\frac{1}{M}q_{11}(\theta)\Big) \ \Big(\frac{1}{M}p^*_{11} + \frac{M-1}{M}\Big) \diff\mu\otimes\nu.
        \label{eq:Eql_integral}
    \end{align}
    Next, we derive the asymptotic behaviour as $M \to \infty$.
    For this we use the Taylor expansion $\log(1+x) = x - x^2/2 + O(x^3)$ (with the fact that $q_{11}(\theta) \in [-1,U-1]$, making the approximation error uniformly bounded) and $M \geq 2$ to obtain
    \begin{align*}
        -\epl (\theta)
        &=
        M^2\int \Big(\frac{q_{11}(\theta)}{M} - \frac{(q_{11}(\theta))^2}{2M^2} + O(M^{-3})\Big) \Big(\frac{1}{M}p^*_{11} + \frac{M-1}{M}\Big) \diff\mu\otimes\nu
        \\&=
        M^2\int 
        \frac{q_{11}(\theta)}{M} \frac{1}{M}p^*_{11} 
        + \frac{q_{11}(\theta)}{M} \frac{(M-1)}{M}
        - \frac{(q_{11}(\theta))^2}{2M^2} \Big(\frac{1}{M}p^*_{11} 
        + \frac{(M-1)}{M}\Big)
        \diff\mu\otimes\nu 
        + O(M^{-1}).
        \intertext{
        We continue by cancelling the second term via $\int q_{11} \diff\mu\otimes\nu = 0$, absorbing more terms into the $O(M^{-1})$ and rearranging the rest
        }
        -\epl (\theta)
        &=
        \frac{1}{2}\int 
        2 q_{11}(\theta) p^*_{11} 
        - (q_{11}(\theta))^2 \left(\frac{1}{M}p^*_{11}  + \frac{(M-1)}{M}\right)
        \diff\mu\otimes\nu 
        + O(M^{-1})
        \\&=
        \frac{1}{2}\int 2q_{11}(\theta) p^*_{11} - (q_{11}(\theta))^2 \diff\mu\otimes\nu + O(M^{-1})
        \\&=
        -\frac{1}{2}\int\Big(q_{11}(\theta) - p^*_{11}\Big)^2 \diff\mu\otimes\nu + \frac{1}{2}\int (p^*_{11})^2 \diff\mu\otimes\nu + O(M^{-1}).
        \intertext{
        Now we re-substitute $q_{11} = p_{11}+1$ to finish the proof
        }
        -\epl (\theta)
        &= 
        -\frac{1}{2}\int\Big(p_{11}(\theta) - p^*_{11}\Big)^2 - 2\Big(p_{11}(\theta) - p^*_{11}\Big) + 1\  \diff \mu\otimes\nu + \frac{1}{2}\int (p^*_{11})^2 \diff\mu\otimes\nu + O(M^{-1})
        \\&=
        -\frac{1}{2}\int\Big(p_{11}(\theta) - p^*_{11}\Big)^2 \diff\mu\otimes\nu + \frac{1}{2}\int (p^*_{11})^2 \diff\mu\otimes\nu - \frac{1}{2} + O(M^{-1}).
        \qedhere
    \end{align*}
\end{proof}

\begin{theorem*}[Recollection of \protect{\Cref{thm:main_detailed}}]
    Let $J,D$ be as in \Cref{lem:Dudley_iid}. There exists an absolute constant $c>0$ such that
    \begin{align*}
        \Pb\left(\sup_{\theta\in\Theta}\abs{\pl(\theta) - \epl(\theta)} \geq 6t\right) \leq 6\exp\left(-c\min\Big\{ \frac{t^2 N}{U^4}, \frac{tN}{U}\Big\}\right) + 6\exp\left(\frac{J}{D}\right)\exp\left(-\frac{ct\sqrt{N}}{D\mathcal{L}U}\right).
    \end{align*}
\end{theorem*}

\noindent 
The full proof will be provided at the end of this section.
To control $\sup_{\theta\in\Theta}\abs{\pl(\theta) - \epl(\theta)}$ we separate $\pl(\theta)$ into three terms, namely
\begin{align}
    \pl(\theta) 
    &= 
    \frac{1}{NM}\sum_{k=1}^N\sum_{i,j=1}^M q_{ij}^k(\theta)  +\pl(\theta) -  \frac{1}{NM}\sum_{k=1}^N\sum_{i,j=1}^M q_{ij}^k(\theta)
    \nonumber
    \\&=
    -\underbrace{\frac{1}{NM}\sum_{k=1}^N\sum_{i,j=1, i\neq j}^M q_{ij}^k(\theta)}_{f_a(\theta)}
    -
    \underbrace{\frac{1}{NM}\sum_{k=1}^N\sum_{i=1}^M q_{ii}^k(\theta)}_{f_b(\theta)}
    +
    \underbrace{\pl(\theta) +  \frac{1}{NM}\sum_{k=1}^N\sum_{i,j=1}^M q_{ij}^k(\theta)}_{f_c(\theta)}
\end{align}
and apply the triangle inequality to get
\begin{align}\label{eq:ql_decompose}
    \sup_{\theta\in\Theta}\abs{\pl(\theta) - \epl(\theta)} \leq \sup_{\theta\in\Theta}\abs{f_a(\theta) - \E f_a(\theta)} + 
    \sup_{\theta\in\Theta}\abs{f_b(\theta) - \E f_b(\theta)} +
    \sup_{\theta\in\Theta}\abs{f_c(\theta) - \E f_c(\theta)}.
\end{align}
Our goal is to obtain an $M$-independent concentration rate for $\pl$, and we separate $\pl$ into three parts for a cleaner presentation. Note that the concentration of $f_a$ requires more effort than the others, since it is the component that naively appears to scale at the rate $O(M)$.
The following lemmata work towards bounding each of the summands above.

\begin{lemma}[Combinatorial interlude]\label{lem:count}
    Let $l \in \N$, and we write $\llbracket 1, l \rrbracket := \{1,\ldots, l\}$ for the first $l$ natural numbers. Define
    \begin{align*}
        S 
        := 
        \{ (z_i)_{i=1}^{2l} \in \llbracket 1,M \rrbracket^{2l} \mid \forall i \, \exists j \neq i : z_j = z_i \}.
    \end{align*}
    This can be interpreted as the set of strings over the alphabet $\llbracket 1, M \rrbracket$ of length $2l$ where no character appears exactly once.
    We have
    \begin{align*}
        \abs{S}
        \leq
        (2 M l)^l.
    \end{align*}
\end{lemma}

\begin{proof}
    Let $B := \{b := (b_i)_{i=1}^{2l} \in \{0,1\}^{2l} \mid \norm{b}_1 = l\}$ be the set of binary strings of length $2l$ with exactly $l$ entries equal to $1$.
    We define a map
    \begin{align*}
        \sigma : 
        B \times \llbracket 1, M \rrbracket^l \times \llbracket 1,l \rrbracket^l \rightarrow S, 
        \quad 
        (b, c, k) \mapsto s
    \end{align*}
    by the following procedure:
    \begin{itemize}
        \item For $i \in \llbracket 1, l \rrbracket$, let $j$ be the $i$-th entry in $b$ equal to $1$. Set $s_j := c_i$.
        \item For $i \in \llbracket 1, l \rrbracket$, let $j$ be the $i$-th entry in $b$ equal to $0$. Set $s_j := c_{k_i}$.
    \end{itemize}
    Next we show that $\sigma$ is surjective and the preimage of any $s \in S$ has at least $2^l$ elements.
    Given an arbitrary $s \in S$, let $\tau$ be a permutation such that $(s_{\tau(i)})_{i=1}^{2l}$ starts with a pair of each letter occurring in $s$ followed by the rest in arbitrary order. This exists since no letter appears exactly once in $s$. In the following, we assume w.l.o.g.~$\tau = \id$.
    We construct preimages of $s \in S$ under $\sigma$ by the following procedure:
    \begin{itemize}
        \item For $i \in \llbracket 1, l \rrbracket$ set 
        $(b_{2i-1}, b_{2i}) \in \{(1,0), (0,1)\}$ 
        arbitrarily.
        
        \item For $i \in \llbracket 1, l \rrbracket$ set 
        $c_i := 
        \begin{cases} 
            s_{2i-1} & \text{if } b_{2i-1} = 1, \\ 
            s_{2i} & \text{if } b_{2i} = 1.
        \end{cases}$
        
        \item For $i \in \llbracket 1, l \rrbracket$ set
        $k_i := 
        \begin{cases}
            c^{-1}(s_{2i-1}) & \text{if } b_{2i-1} = 0, \\
            c^{-1}(s_{2i}) & \text{if } b_{2i} = 0
        \end{cases}$
        
        where $c^{-1}(a)$ denotes an arbitrary index $j \in \llbracket 1, l \rrbracket$ such that $c_j = a$ (which exists by assumption and construction).
    \end{itemize}
    Applying the procedure defining $\sigma$ to this, it is easy to see that the constructed values are indeed preimages and since $(b_{2i-1}, b_i)$ are chosen independently and arbitrarily from two possibilities each, we have $2^l$ different preimages.
    Therefore
    \begin{align*}
        \abs{S} \leq \frac{\abs{B} M^l l^l}{2^l} \leq \frac{2^{2l} M^l l^l}{2^l} &= (2 M l)^l.
        \qedhere
    \end{align*}
\end{proof}

\begin{lemma}\label{lem:fa}
$f_a(\theta)$ is a $\psi_1$-process with respect to $\frac{c\lip  }{\sqrt{N}}\norm{\cdot}$, in which $c>0$ is some absolute constant and $\norm{\cdot}$ is the original norm on $\Theta$. 
\end{lemma}
\begin{proof}
    Let $F_k(\theta) := \frac{1}{M}\sum_{i\neq j}^M q_{ij}^k(\theta)$, then $f_a(\theta) = \frac{1}{N}\sum_{k=1}^N F_k(\theta)$. Note that $F_k(\theta)$ has mean
    \begin{align*}
        \E F_k(\theta) 
        = 
        \frac{1}{M} \sum_{i\neq j}^M\int q_{ij} \prod_{m = 1}^M p^*_{mm} \diff 
        (\mu\otimes\nu)^{\otimes M}
        = 
        \frac{1}{M} \sum_{i\neq j}^M\int q_{ij} \diff 
        (\mu\otimes\nu)^{\otimes M} 
        = 
        0.
    \end{align*}
    For even $l\in\mathbb{N}$ we get
    \begin{align*}
        \E\left(\abs{F_1(\theta) - F_1(\theta')}^l\right)
        = 
        \E\left(\left(F_1(\theta) - F_1(\theta')\right)^l\right)
        =
        \frac{1}{M^l}\sum_{i_1\neq j_1}^M \cdots \sum_{i_l\neq j_l}^M\E\left(\prod_{\alpha=1}^l\left(q_{i_\alpha j_\alpha}^1(\theta) - q_{i_\alpha j_\alpha}^1(\theta')\right)\right).
    \end{align*}
    We examine the expectation term. By definition we have
    \begin{align*}
        \E\left(\prod_{\alpha=1}^l\left(q_{i_\alpha j_\alpha}^1(\theta) - q_{i_\alpha j_\alpha}^1(\theta')\right)\right)
        &= \int \prod_{\alpha=1}^l\left(q_{i_\alpha j_\alpha}(\theta) - q_{i_\alpha j_\alpha}(\theta')\right) \prod_{m=1}^M p^*_{mm} \ \diff(\mu\otimes\nu)^{\otimes M}.
    \end{align*}
    Given a set of indices as above, if there exists some $i_{\beta}$ that is unique, i.e.~$i_{\beta} \not \in \{i_ {\alpha} \mid \alpha: \alpha \neq \beta\} \cup \{j_{\alpha}\mid\alpha\}$, we can first perform the integral over $y_{i_{\beta}}$ (which only appears in the $p^{\ast}$-term, yielding $1$) and then perform the integral over $x_{i_{\beta}}$ (which only appears in one $q$-term, yielding $0$). This factor then implies that the whole integral is $0$. Similarly, if there is some unique $j_{\beta}$, the integral is $0$.
    We conclude that if there exists any value that only appears once in $(i_1,\cdots,i_l,j_1,\cdots,j_l)$ then the integral is zero. 
    Using $S,\mathcal{L}$ as defined in \Cref{lem:count} and \Cref{def:candidate} together with the equations above, we get
    \begin{align*}
        \E\left(\abs{F_1(\theta) - F_1(\theta')}^l\right)
        \leq
        \frac{1}{M^l} \abs{S} \sup_{x,y}\abs{q^\theta(x,y) - q^{\theta'}(x,y)}^l
        \leq
        \frac{1}{M^l} \abs{S} \lip^l \norm{\theta - \theta'}^l
        \leq
        \left(2 \lip \norm{\theta - \theta'} l\right)^l.
    \end{align*}
    Now for $l\in \mathbb{N}$ being odd, apply Jensen's inequality
    \begin{align*}
        \E\left(\abs{F_1(\theta) - F_1(\theta')}^l\right) 
        \leq
        \left(\E\left(\abs{F_1(\theta) - F_1(\theta')}^{l+1}\right)\right)^{\frac{l}{l+1}}
        \leq
        \left(2 \lip \norm{\theta - \theta'} (l+1)\right)^l
        \leq
        \left(4 \lip \norm{\theta - \theta'} l\right)^l.
    \end{align*}
    By \Cref{lem:psi_norm} we have $\norm{F_1(\theta) - F_1(\theta')}_{\psi_1} \leq 8e\lip   \norm{\theta - \theta'}$. Since $F_k$ are i.i.d.~and $f_a(\theta) = \frac{1}{N}\sum_{k=1}^N F_k$, by \Cref{lem:psi_norm_iid} we have $\norm{f_a(\theta) - f_a(\theta')}_{\psi_1} \leq \frac{c\lip  }{\sqrt{N}} \norm{\theta - \theta'}$ for some absolute constant $c>0$.
\end{proof}
\begin{lemma}\label{lem:fb}
    $f_b(\theta) - \E f_b(\theta)$ is a $\psi_1$-process with respect to $\frac{c\lip  }{\sqrt{N}} \norm{\cdot}$ for some absolute constant $c>0$.
\end{lemma}
\begin{proof}
    Note that $\abs{\frac{1}{M}\sum_{i=1}^M \left(q^1_{ii}(\theta) - q^1_{ii}(\theta')\right)} \leq \lip   \norm{\theta - \theta'}$ and apply \Cref{lem:trivial}.
\end{proof}

\begin{lemma}\label{lem:fc}
    $f_c(\theta) - \E f_c(\theta)$ is a $\psi_1$-process with respect to $\frac{c\lip U  }{\sqrt{N}} \norm{\cdot}$ for some absolute constant $c>0$.
\end{lemma}
\begin{proof}
    We have
    \begin{align*}
        \nabla_\theta \sum_{i,j=1}^M\log\left(\frac{1}{M}p^1_{ij}(\theta) + \frac{M-1}{M}\right) 
        -
        \nabla_\theta \frac{1}{M}\sum_{i,j=1}^M q_{ij}^1(\theta)
        &= 
        \sum_{i,j=1}^M\left(\frac{\frac{1}{M}\nabla_\theta p^1_{ij}(\theta)}{\frac{1}{M}p^1_{ij}(\theta) + \frac{M-1}{M}} - \frac{1}{M}\nabla_\theta p_{ij}^1(\theta)\right)
        \\&=
        \frac{1}{M^2}\sum_{i,j=1}^M \frac{(1 - p_{ij}^1(\theta)) \nabla_\theta p_{ij}^1(\theta)}{\frac{1}{M}p^1_{ij}(\theta) + \frac{M-1}{M}}.
    \end{align*}
    Using the assumption $M \geq 2$, this implies
    \begin{align*}
        \norm{-\nabla_\theta \sum_{i,j=1}^M\log\left(\frac{1}{M}p^1_{ij}(\theta) + \frac{M-1}{M}\right) + \nabla_\theta \frac{1}{M}\sum_{i,j=1}^M q_{ij}^1(\theta)}_{2} \leq 2 \lip U.
    \end{align*}
    Applying \Cref{lem:trivial} finishes the proof.
\end{proof}
\begin{lemma}\label{lem:fa_fix}
    For fixed $\theta_0\in \Theta$, there is an absolute constant $c>0$ such that
    \begin{align*}
        \Pb\left(\abs{f_a(\theta_0)} \geq t\right) \leq 2 \exp\left(-c\min\Big\{\frac{t^2 N}{U^2}, \frac{tN}{U}\Big\}\right).
    \end{align*}
\end{lemma}
\begin{proof}
    The proof works very similarly to that of \Cref{lem:fa}.
    Let $F_k(\theta_0) = \frac{1}{M}\sum_{i\neq j} q_{ij}^k(\theta_0)$. For even $l\in \mathbb{N}$ we get
    \begin{align*}
        \E\left(\abs{F_1(\theta_0)}^l\right) 
        &= 
        \frac{1}{M^l}
        \sum_{i_1\neq j_1}\cdots\sum_{i_l\neq j_l}
        \E\left(\prod_{\alpha = 1}^l q^1_{i_\alpha j_\alpha}\right)
        \leq 
        \left(2l\sup_{x,y,\theta}\abs{p^\theta (x,y) - 1}\right)^{\hspace{-2pt}l} 
        \leq \left(2lU\right)^l.
    \end{align*}
    Using Jensen to get the bound for odd $l$, then by \Cref{lem:psi_norm} we have $\norm{F_1}_{\psi_1} \leq 8eU$. 
    Now apply \Cref{thm:Bernstein} to finish the proof.
\end{proof}
\begin{lemma} \label{lem:fb_fix}
    For fixed $\theta_0\in \Theta$ we have
    \begin{align*}
        \Pb\left(\abs{f_b(\theta_0) - \E f_b(\theta_0)} \geq t\right)\leq 2\exp\left(-\frac{t^2 N}{2U^2}\right).
    \end{align*}
\end{lemma}
\begin{proof}
    Note that $\abs{\frac{1}{M}\sum_{i=1}^M q_{ii}^1(\theta)} \leq \sup_{x,y,\theta}\abs{p^\theta (x,y) - 1} \leq U$. Apply \Cref{lem:Hoeffding}.
\end{proof}
\begin{lemma} \label{lem:fc_fix}
    For fixed $\theta_0\in \Theta$ we have
    \begin{align*}
        \Pb\left(\abs{f_c(\theta_0) - \E f_c(\theta_0)} \geq t\right)\leq 2\exp\left(-\frac{t^2 N}{8U^4}\right).
    \end{align*}
\end{lemma}
\begin{proof}
    Note that for any $x > 0$ there is a $\xi = \xi(x)$ between $1$ and $1+x$ such that $\log(1+x) = x - \frac{x^2}{2\xi^2}$. Therefore
    \begin{align*}
        \abs{\sum_{i,j=1}^M \left(\log\Big(1 + \frac{1}{M}q_{ij}^1(\theta)\Big) - \frac{1}{M}q_{ij}^1(\theta)\right)}
        &= 
        \abs{\frac{1}{2M^2}\sum_{i,j=1}^M \frac{(q_{ij}^1(\theta))^2}{\xi^2}}
        \leq 2\sup_{x,y,\theta}\abs{p^\theta(x,y)-1}^2 \leq 2U^2
    \end{align*}
    in which $\xi$ is between $1$ and $1+\frac{1}{M}q_{ij}^1$, and since $M\geq 2$ we have $\inf\abs{\xi} \geq 1/2$. Apply \Cref{lem:Hoeffding}.
\end{proof}

\begin{proof}[Proof of \Cref{thm:main_detailed}]
    By \Cref{lem:fa_fix} $f_a$ is bounded at a single point with high probability, by \Cref{lem:fa} it is a $\psi_1$-process. \Cref{lem:Dudley_iid} then bounds pairwise differences with high probability. Finally, \Cref{lem:aux} bounds $f_a$ uniformly with high probability. We obtain for some absolute constant $c$,
    \begin{align*}
        \Pb\bk{\sup_\theta\abs{f_a(\theta)} \geq 2t}
        &\leq
        2 \exp\left(-c\min\Big\{\frac{t^2 N}{U^2}, \frac{tN}{U}\Big\}\right)
        +
        2\exp\left(\frac{J}{D}\right)\exp\left(-\frac{ct\sqrt{N}}{D\mathcal{L}}\right).
    \end{align*}
    Similarly for $f_b(\theta) - \E f_b(\theta)$ (using Lemmata \ref{lem:fb} and \ref{lem:fb_fix}) and $f_c(\theta) - \E f_c(\theta)$ (using Lemmata \ref{lem:fc} and \ref{lem:fc_fix}) we get
    \begin{align*}
        \Pb\bk{\sup_\theta\abs{f_b(\theta) - \E f_b(\theta)} \geq 2t}
        &\leq 
        2\exp\left(-\frac{t^2 N}{2U^2}\right)
        +
        2\exp\left(\frac{J}{D}\right)\exp\left(-\frac{ct\sqrt{N}}{D\mathcal{L}}\right),
        \\
        \Pb\bk{\sup_\theta\abs{f_c(\theta) - \E f_c(\theta)} \geq 2t}
        &\leq
        2\exp\left(-\frac{t^2 N}{8U^4}\right)
        +
        2\exp\left(\frac{J}{D}\right)\exp\left(-\frac{ct\sqrt{N}}{D\mathcal{L}U}\right).
    \end{align*}
    Use \Cref{eq:ql_decompose} with union bound to finish the proof.
\end{proof}

\subsection[Concentration of gradient of f\_M\textasciicircum{}N(theta)]{Concentration of $\nabla\pl(\theta)$}\label{sec:concentrationGradient}

In this subsection we prove a similar bound on $\sup_\theta ||\nabla\pl(\theta) - \nabla\epl(\theta)||_2$ for the gradient by using the same strategy as in Section \ref{sec:concentration}.
Since $\Theta$ is finite-dimensional, we can bound the norm entry-wise. 
We abbreviate $\frac{\partial}{\partial \theta_l}$ as $\partial_l$. 
As before, we first decompose
\begin{align*}
    \partial_l\pl(\theta) &= -\frac{1}{N}\sum_{k=1}^N \sum_{i,j=1}^M \frac{\frac{1}{M}\partial_l p_{ij}^k(\theta)}{\frac{1}{M}p_{ij}^k(\theta) + \frac{M-1}{M}} \ \Big(1 - \frac{1}{M} + \frac{1}{M} - \frac{1}{M}p_{ij}^k(\theta) + \frac{1}{M}p_{ij}^k(\theta)\Big)
    \\&=
    -\underbrace{\frac{1}{NM}\sum_{k=1}^N \sum_{i\neq j}^M \partial_l p_{ij}^k(\theta)}_{f'_a(\theta)}
    -\underbrace{\frac{1}{NM}\sum_{k=1}^N \sum_{i=1}^M \partial_l p_{ii}^k(\theta)}_{f'_b(\theta)}
    +\underbrace{\frac{1}{NM^2}\sum_{k=1}^N \sum_{i, j=1}^M \frac{(p_{ij}^k(\theta) - 1)\partial_l p_{ij}^k(\theta)}{\frac{1}{M}p_{ij}^k(\theta) + \frac{M-1}{M}}}_{f'_c(\theta)}
\end{align*}
\noindent
and then use the triangle inequality to get
\begin{align*}
    \sup_{\theta\in\Theta}\abs{\partial_l\pl(\theta) - \partial_l\epl(\theta)} 
    \leq 
    \sup_{\theta\in\Theta}\abs{f'_a(\theta) - \E f'_a(\theta)} + 
    \sup_{\theta\in\Theta}\abs{f'_b(\theta) - \E f'_b(\theta)} +
    \sup_{\theta\in\Theta}\abs{f'_c(\theta) - \E f'_c(\theta)}.
\end{align*}
\noindent
The proofs of the following lemmata, unless otherwise stated, are exactly the same as the proofs of the corresponding lemmata in the previous section, obtained by simply replacing $q_{ij}^k(\theta)$ with $\partial_l p_{ij}^k(\theta)$. 
Note that $\int \partial_l p_{ij}^k(\theta) \diff \mu(x) = \partial_l \int p_{ij}^k \diff\mu(x) = \partial_l 1 = 0$ and the same for the other marginal.
We are allowed to swap the integration and differentiation due to \Cref{thm:interchange} since $p_{ij}^k$ and $\partial_l p_{ij}^k$ are uniformly bounded by \Cref{def:candidate}.
\begin{lemma}\label{lem:2fa}
    $f'_a(\theta)$ is a $\psi_1$-process with respect to $\frac{c\lip'}{\sqrt{N}}\norm{\cdot}$, in which $c>0$ is some absolute constant and $\norm{\cdot}$ is the original norm on $\Theta$. 
\end{lemma}

\begin{lemma}\label{lem:2fb}
    $f'_b(\theta) - \E f'_b(\theta)$ is a $\psi_1$-process with respect to $\frac{c\lip'  }{\sqrt{N}} \norm{\cdot}$ for some absolute constant $c>0$.
\end{lemma}

\begin{lemma}\label{lem:2fc}
    $f'_c(\theta) - \E f'_c(\theta)$ is a $\psi_1$-process w.r.t.~$\frac{cU(\lip^2+\lip')  }{\sqrt{N}} \norm{\cdot}$ for some absolute constant $c>0$. 
\end{lemma}
\begin{proof}
We proceed analogously to the proof of \Cref{lem:fc}. First we bound the Lipschitz constant
\begin{align*}
\norm{\nabla_\theta\frac{1}{M^2}\sum_{i, j=1}^M \frac{q_{ij}\partial_l p_{ij}(\theta)}{\frac{1}{M}p_{ij}(\theta) + \frac{M-1}{M}}}_2
&\leq 
\sup_{x_1,y_1}\norm{\nabla_\theta \frac{q_{11}(\theta)\partial_l p_{11}(\theta)}{\frac{1}{M}p_{11}(\theta) + \frac{M-1}{M}}}_2
\\&\leq
\sup_{x_1,y_1}\norm{\frac{\nabla_\theta[q_{11}(\theta)\partial_l p_{11}(\theta)]}{\frac{1}{M}p_{11}(\theta) + \frac{M-1}{M}}
-
\frac{q_{11}(\theta)\partial_l p_{11}(\theta)\nabla_\theta(\frac{1}{M}p_{11}(\theta) + \frac{M-1}{M})}{(\frac{1}{M}p_{11}(\theta) + \frac{M-1}{M})^2}
}_2
\\&\leq
2\lip^2 + 2U\lip' + \frac{4U\lip^2}{M}
\leq
4U(\lip^2 + \lip')
\end{align*}
and then we apply \Cref{lem:trivial}.
\end{proof}

\begin{lemma}\label{lem:2fa_fix}
    For fixed $\theta_0\in \Theta$, there is an absolute constant $c>0$ such that
    \begin{align*}
        \Pb\left(\abs{f'_a(\theta_0)} \geq t\right) \leq 2 \exp\left(-c\min\Big\{\frac{t^2 N}{\lip^2}, \frac{tN}{\lip}\Big\}\right).
    \end{align*}
\end{lemma}

\begin{lemma} \label{lem:2fb_fix}
    For fixed $\theta_0\in \Theta$ we have
    \begin{align*}
        \Pb\left(\abs{f'_b(\theta_0) - \E f'_b(\theta_0)} \geq t\right)\leq 2\exp\left(-\frac{t^2 N}{2\lip^2}\right).
    \end{align*}
\end{lemma}

\begin{lemma} \label{lem:2fc_fix}
    For fixed $\theta_0\in \Theta$ we have
    \begin{align*}
        \Pb\left(\abs{f'_c(\theta_0) - \E f'_c(\theta_0)} \geq t\right)\leq 2\exp\left(-\frac{t^2 N}{8U^2\lip^2}\right).
    \end{align*}
\end{lemma}
\begin{proof}
    The upper bound for $|(p^\theta-1)\partial_lp^\theta|$ is $U\lip$, the lower bound for $|\frac{1}{M}p^\theta + \frac{M-1}{M}|$ is $1/2$ since $M\geq 2$. Apply \Cref{lem:Hoeffding}.
\end{proof}

\begin{theorem}\label{thm:2main_detailed}
    Let $J,D$ be same as in \Cref{lem:Dudley_iid}. There exists an absolute constant $c>0$ such that
    \begin{align*}
        \Pb\left(\sup_{\theta\in\Theta}\abs{\partial_l\pl(\theta) - \partial_l\epl(\theta)} \geq 6t\right) 
        &\leq 
        6\exp\left(-c\min\Big\{ \frac{t^2 N}{U^2\lip^2}, \frac{tN}{\lip}\Big\}\right) + 6\exp\left(\frac{J}{D}\right)\exp\left(-\frac{ct\sqrt{N}}{D(\lip^2+\lip') U}\right).
    \end{align*}
\end{theorem}

\begin{theorem}\label{thm:grad_bound}
    There exist constants $C_1,C_2>0$ independent of $M,N$ such that
    \begin{align*}
        \Pb\left(\sup_{\theta\in\Theta}\norm{\nabla\pl(\theta) - \nabla\epl(\theta)}_2 \geq t\right) \leq C_1\exp\Big(-C_2\min\{t^2N, t\sqrt{N}\}\Big)
    \end{align*}
\end{theorem}
\begin{proof}
    We have
    \begin{align*}
        \Pb\left(\sup_{\theta\in\Theta}\norm{\nabla\pl(\theta) - \nabla\epl(\theta)}_2 \geq t\right)
        &\leq
        \Pb\left(\sup_{\theta\in\Theta}\norm{\nabla\pl(\theta) - \nabla\epl(\theta)}_\infty \geq \frac{t}{\sqrt{\dim(\Theta)}}\right)
        \\&\leq
        \sum_{m=1}^M\Pb\left(\sup_{\theta\in\Theta}\abs{\partial_l\pl(\theta) - \partial_l\epl(\theta)} \geq \frac{t}{\sqrt{\dim(\Theta)}}\right).
    \end{align*}
    Apply \Cref{thm:2main_detailed} and clean up.
\end{proof}

\subsection{Convergence of the estimator}\label{subsec:estimator_convergence}
In this subsection we provide an explicit convergence rate for our estimator $\mms$.
\begin{lemma}\label{lem:curvatrure_convergence}
    $\nabla^2 \epl(\theta) \to \frac{1}{2}\nabla^2 \norm{p^\theta - p^*}^2_{L^2(\mu\otimes\nu)} = \nabla^2 \lpl(\theta)$ uniformly as $M\to\infty$.
\end{lemma}
\begin{proof}
By \Cref{eq:Eql_integral} we have $\epl(\theta) = -M^2\int \log(\frac{1}{M}p^\theta + \frac{M-1}{M}) (\frac{1}{M}p^* + \frac{M-1}{M}) \diff\mu\otimes\nu$.
Note that by \Cref{def:candidate} the value as well as any first and second derivatives of $p^{\theta}$ are uniformly bounded. A simple calculation shows that the same holds for $\log(\frac{1}{M} p^{\theta} + \frac{M-1}{M})$ for $M \geq 2$. With this, \Cref{thm:interchange} allows us to exchange differentiation and integration, therefore
    \begin{align*}
        -\nabla^2 \epl(\theta) 
        &= 
        M^2\int \left(\frac{\frac{1}{M}\nabla^2 p^\theta}{\frac{1}{M}p^\theta + \frac{M-1}{M}} - \frac{\frac{1}{M^2}(\nabla p^\theta) \otimes (\nabla p^\theta)}{\left(\frac{1}{M}p^\theta + \frac{M-1}{M}\right)^2}\right) \left(\frac{1}{M}p^* + \frac{M-1}{M}\right) \diff\mu\otimes\nu
        \\&=
        M^2\int 
        \left(
        \frac{\frac{1}{M}\nabla^2 p^\theta}{\frac{1}{M}p^\theta + \frac{M-1}{M}} 
        \right) 
        \left(
        \frac{1}{M}(p^* - p^{\theta}) + \left(\frac{1}{M}p^{\theta} + \frac{M-1}{M}\right)
        \right) \diff\mu\otimes\nu
        \\&
        \qquad
        -M^2\int 
        \left(
        \frac{\frac{1}{M^2}(\nabla p^\theta)^{\otimes 2}}
        {\left(\frac{1}{M}p^\theta + \frac{M-1}{M}\right)^2}
        \right) 
        \left(
        \frac{1}{M}(p^* - p^{\theta}) + \left(\frac{1}{M}p^{\theta} + \frac{M-1}{M}\right)
        \right) \diff\mu\otimes\nu
        \\&=
        \int\frac{(p^* - p^\theta)\nabla^2 p^\theta}{\frac{1}{M}p^\theta + \frac{M-1}{M}} + M\nabla^2 p^\theta
        -\frac{1}{M}\frac{(p^*-p^\theta)(\nabla p^\theta)^{\otimes 2}}{\left(\frac{1}{M}p^\theta + \frac{M-1}{M}\right)^2}
        - \frac{(\nabla p^\theta)^{\otimes 2}}{\frac{1}{M}p^\theta + \frac{M-1}{M}}
        \diff\mu\otimes\nu.
        \intertext{
        Using the identities $\frac{1}{\frac{1}{M}p^\theta + \frac{M-1}{M}} = 1 + \frac{1}{M}\frac{1-p^\theta}{\frac{1}{M}p^\theta + \frac{M-1}{M}}$ and $\int\nabla^2p^\theta \diff \mu\otimes\nu = 0$ we get
        }
        -\nabla^2 \epl(\theta) 
        &=
        \int \left(p^* - p^\theta\right)\nabla^2 p^\theta
        \left(1 + \frac{1}{M}\left(\frac{1 - p^{\theta}}{\frac{1}{M} p^{\theta} + \frac{M-1}{M}}\right)\right)
        \diff\mu\otimes\nu
        \\&\qquad-
        \int 
        (\nabla p^\theta)^{\otimes 2}
        \left(
            1 + \frac{1}{M}\left(
                \frac{1-p^\theta}{\frac{1}{M}p^\theta + \frac{M-1}{M}} + \frac{p^* - p^\theta}{\left(\frac{1}{M}p^\theta + \frac{M-1}{M}\right)^2}
            \right)
        \right)
        \diff \mu\otimes\nu.
    \end{align*}
    Note that $\frac{1}{2}\nabla^2 \norm{p^\theta - p^*}^2_{L^2(\mu\otimes\nu)} = \int (\nabla p^\theta)^{\otimes 2} - (p^*-p^\theta)\nabla^2 p^\theta \diff\mu\otimes\nu$, therefore 
    \begin{align*}
        \nabla^2 \epl(\theta) = \frac{1}{2}\nabla^2 \norm{p^\theta - p^*}^2_{L^2(\mu\otimes\nu)} + \frac{1}{M} A(\theta)
    \end{align*}
    for $A(\theta)$ being some uniformly (with respect to $\theta$) bounded operator. 
\end{proof}

\noindent Then we are able to control the curvature of $\epl(\theta)$ for large $M$ under the following assumption.
\begin{assumption}\label{ass:L2}
    Assume that
    $\lmms \in \Theta$ is the unique minimiser of $\lpl$ over $\Bar\Theta$, and there exist some $\tau, r > 0$ such that the smallest eigenvalue of $\nabla^2 \lpl(\theta)$ is larger than $\tau$ for all $\theta\in\Theta$ with $\norm{\theta - \lmms} \leq r$.
\end{assumption}

\begin{theorem*}[Recollection of \protect{\Cref{thm:main-estimator}}]
    There exists an $M_0 \in \mathbb{N}$ such that $\forall M\geq M_0$ the following hold.
    \begin{enumerate}
        \item The minimiser $\emms$ of $\epl$ is unique.
        \item There exist some $C_1,C_2,t_0 > 0$ s.t.~for all $t \leq t_0$ and for any $\mms\in\argmin_{\Bar\Theta} \pl$ we have
    \begin{align*}
        \Pb\left(\norm{\mms- \emms}_2 \geq t\right) \leq C_1\exp\left(-C_2\min\left\{t^2 N, t\sqrt{N}\right\}\right).
    \end{align*}
    \end{enumerate}
\end{theorem*}
\begin{proof}
    Suppose that the maximal pairwise distance between $\argmin_{\Bar\Theta} \epl$ and $\{\lmms\}$ did not converge to $0$ as $M\to\infty$, i.e.~there exists a sequence $(\tilde\theta_M)_M$ such that $\tilde\theta_M \in\argmin_{\Bar\Theta}\epl$ and $\limsup_M||\lmms - \tilde\theta_M|| > 0$. 
    By compactness of $\Bar\Theta$ there exists a convergent subsequence of minimisers of $\epl$, which does not converge to $\lmms$. 
    By \Cref{ass:L2} the latter is the unique minimizer of $\lpl$. 
    Therefore, this is a contradiction, since by \Cref{thm:mean_pseudo} we have that $\epl$ uniformly converges to $\lpl$ as $M\to\infty$.

    Furthermore, by \Cref{ass:L2} and \Cref{lem:curvatrure_convergence} there exists a closed ball $B(\lmms,r)$ around $\lmms$ with radius $r$ and some $M_1\in\mathbb{N}$ such that for all $M > M_1$, $\epl$ is strictly convex in $B(\lmms,r)$. 
    By the first part of the proof, the set $\argmin \epl$ will also be contained in $B(\lmms,r)$ for all $M$ larger than some $M_2>0$. Therefore for $M > \max(M_1,M_2)$, $\epl$ has a unique minimiser, denoted by $\emms$.

    By uniqueness of the minimizer $\lmms$ of the continuous function $\lpl$ and compactness of $\Bar\Theta$, for every $\epsilon > 0$ we have $\inf \{ \lpl(\theta) \mid \theta \in \Theta,  \norm{\theta - \lmms} \geq \epsilon\} > \lpl(\lmms)$, in particular $\forall \epsilon > 0 \,\exists \delta > 0 : \norm{\theta - \lmms} \geq \epsilon/2 \Rightarrow \lpl(\theta) \geq \lpl(\lmms)+2\delta$.
    Recall that we have $\emms\to\lmms$ and uniform convergence of $\epl\to\lpl$ as $M\to\infty$ (\Cref{thm:mean_pseudo}), therefore there exists an $M_3 > 0$ (depending on $\epsilon$ and $\delta$) such that for all $M > M_3$ and $\theta \in \Bar\Theta$ we get
    \begin{align}\label{eq:curvature}
        \norm{\theta - \emms} \geq \epsilon \Rightarrow
        \norm{\theta - \lmms} \geq \epsilon/2 \Rightarrow
        \lpl(\theta) \geq \lpl(\lmms) + 2\delta \Rightarrow
        \epl(\theta) \geq \epl(\emms) + \delta.
    \end{align}
    Let $\epsilon < r/2$.
    For $\delta' := \sup_\theta |\pl(\theta) - \epl(\theta)|$ we have for any $\mms\in\argmin_{\Bar\Theta}\pl$
    \begin{align*}
        \epl(\mms) - \delta' \leq \pl(\mms) \leq \pl(\emms) \leq \epl(\emms) + \delta'.
    \end{align*}
    If $\delta' < \delta/2$ this leads to $\epl(\mms) \leq \epl(\emms) + 2\delta' < \epl(\emms) + \delta$ and by contrapositive of \eqref{eq:curvature} we have $\norm{\mms - \emms} < \epsilon < r/2$. 
    By \Cref{lem:curvatrure_convergence} and \Cref{ass:L2}, there exists some $M_4 > 0$ such that for all $M\geq M_4$ we have that the smallest eigenvalue of $\nabla^2\epl(\theta)$ is larger than $\tau/2$ for $\theta \in B(\emms, r/2)$.
    Consider the path $\gamma(t) = t\mms + (1-t)\emms$ with $t \in [0,1]$.
    We get
    \begin{align*}
        \norm{\nabla \epl(\mms) - \nabla \pl(\mms)}_2
        &=
        \norm{\nabla \epl(\mms) - \nabla \epl(\emms)}_2 
        \\&=
        \norm{\nabla \epl(\gamma(1)) - \nabla \epl(\gamma(0))}_2 
        \\&=
        \norm{\int_0^1 \nabla^2\epl(\gamma(t)). \dot{\gamma}(t) \diff t}_2
        \\&=
        \norm{\int_0^1 \nabla^2\epl(\gamma(t)). (\mms - \emms) \diff t}_2
        \\&\geq
        \Big\langle \frac{\mms - \emms}{\norm{\mms - \emms}_2},\int_0^1 \nabla^2\epl(\gamma(t)). (\mms - \emms) \diff t\Big\rangle
        \\&=
        \int_0^1 \Big\langle \frac{\mms - \emms}{\norm{\mms - \emms}_2},\nabla^2\epl(\gamma(t)). (\mms - \emms) \Big\rangle\diff t
        \\&\geq
        \frac{\tau}{2}\norm{\mms - \emms}_2.
    \end{align*}

    \noindent
    Combining everything above, we can conclude
    \begin{align*}
        \sup_\theta |\pl(\theta) - \epl(\theta)| < \frac\delta 2
        \Rightarrow
        ||\mms - \emms||_2 \leq \frac{2}{\tau}\sup_\theta \norm{\nabla \epl(\theta) - \nabla \pl(\theta)}_2.
    \end{align*}
    Apply \Cref{thm:main_detailed,thm:grad_bound} and after some straight forward cleaning we prove the claim.
\end{proof}

\section{Numerical examples}\label{sec:num}
In this section we give two synthetic numerical examples that illustrate the key properties of our estimator in accordance with the theoretical results.\footnote{Code available at \url{https://github.com/OTGroupGoe/BatchedBrokenSamples}}

\subsection{Points colocalisation on 2D Torus}
\begin{figure}
    \centering
    \includegraphics[width = 0.95\textwidth]{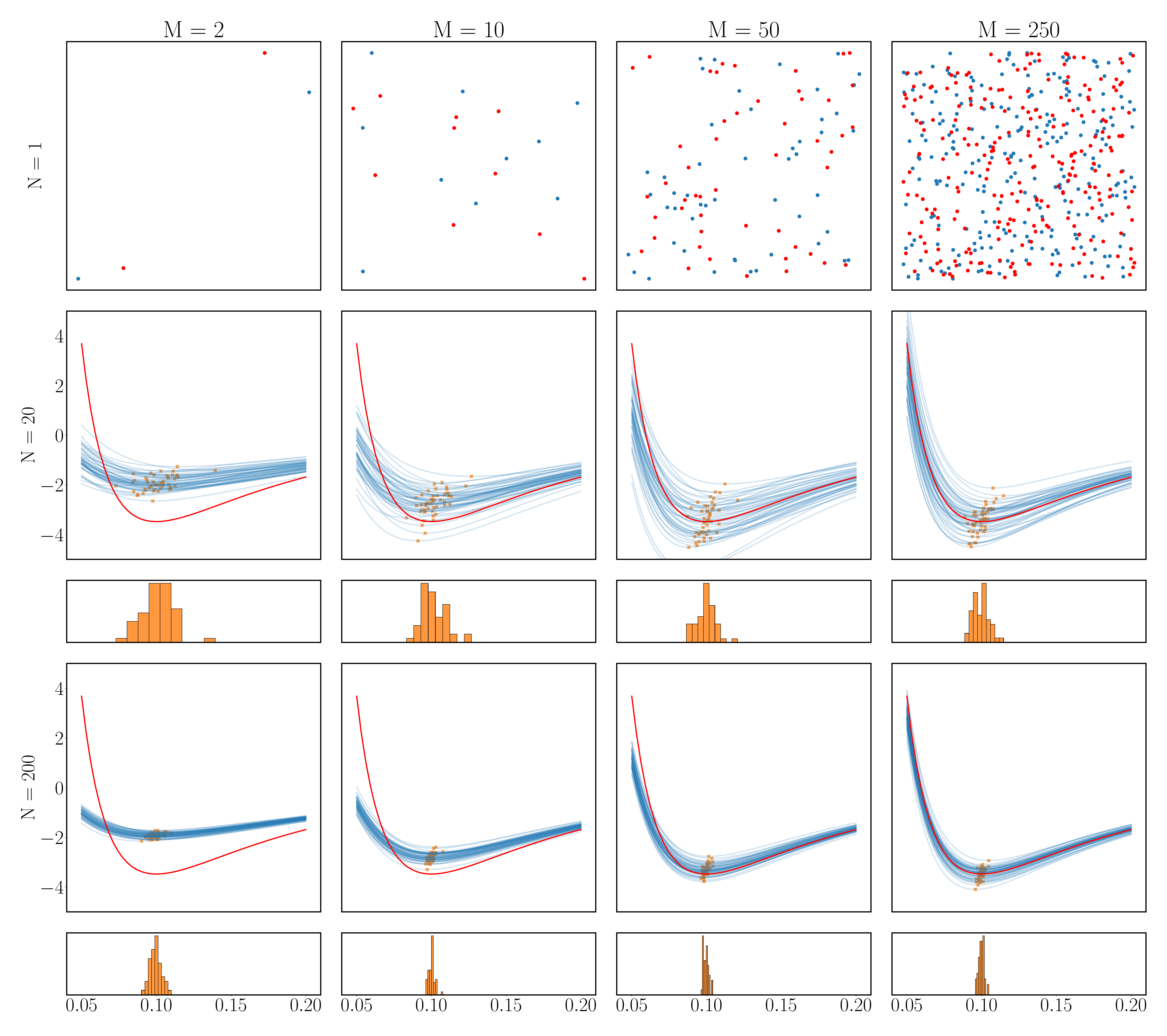}
    \caption{
    Numerical experiment for $\sigma^* = 0.1$, the first row shows an example point cloud from a single batch. 
    Each blue line depicts $\pl(\sigma)$ calculated from samples with $M$ and $N$ as denoted on the corresponding column / row. 
    There are $50$ independent samples per plot.
    Orange points denote the minima and the histograms below show their distribution.
    The red line is $\lpl(\sigma) = \frac{1}{2}||p^\sigma - p^{\sigma^*}||^2_{L^2} + \frac{1}{2} - \frac{1}{2}||p^{\sigma^*}||_{L^2}^2$.
    }
    \label{fig:the_plot}
\end{figure}

\begin{figure}
    \centering
    \includegraphics[width = \textwidth]{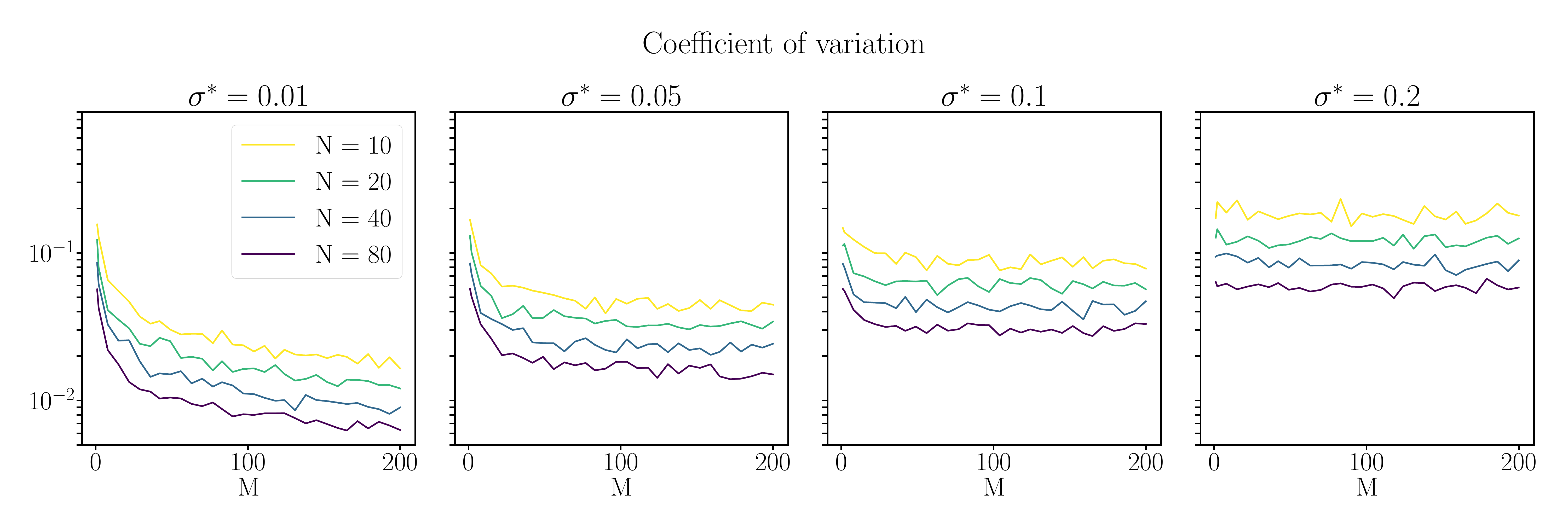}
    \caption{The coefficient of variation of $\sigma_M^N$, computed from $100$ simulations, for varying values of $\sigma^*, N$, and $M$.}
    \label{fig:cv_plot}
\end{figure}

Let $\mathbb{X} = \mathbb{Y} = \mathbb{R}^2/\mathbb{Z}^2$ be the $2$-torus and $\pi$ be defined by
\begin{align*}
(X,Y)\sim \pi \quad \Leftrightarrow \quad
\begin{cases}
    X &\sim \mathcal{U}(\mathbb{R}^2/\mathbb{Z}^2)\\
    Y|X = x &\sim \Tilde{\mathcal{N}}(x, \sigma^2 I).
\end{cases}
\end{align*}
Here $\mathcal{U}$ and $\Tilde{\mathcal{N}}(x, \sigma^2 I)$ denote the uniform distribution and the wrapped $2D$-Gaussian distribution on the $2$-torus centred at $x$ with covariance matrix $\sigma^2 I$ respectively, and $\sigma$ is the parameter we want to estimate. Note that in this setting both of the marginals of $\pi$ are the uniform distribution on the torus. Therefore the density is simply
\begin{align*}
    p^\sigma(x,y) = \sum_{k\in \mathbb{Z}^2}\frac{1}{2\pi \sigma^2} \exp\Big(-\frac{\norm{x - y + k}^2_{\mathbb{R}^2}}{2\sigma^2}\Big).
\end{align*}

We use the ground truth value $\sigma^* = 0.1$ to generate the samples with $M \in \{2, 10, 50, 250\}$ and $N \in \{20, 200\}$. 
In the first row of \Cref{fig:the_plot}, we illustrate a single broken sample batch for varying $M$. 
The red points ($X$) and blue points ($Y$) are generated in pairs, but the pairing information is not observable from these samples. Although one might guess the pairing when $M$ is small, this is clearly impossible when $M$ is large.

In the second row of \Cref{fig:the_plot}, we show simulation results for $N = 20$ and varying $M$. 
The red curves show the limit $\lpl(\sigma) = \tfrac{1}{2}\|p^\sigma - p^{\sigma^*}\|_{L^2}^2 + \tfrac{1}{2} - \tfrac{1}{2}\|p^{\sigma^*}\|_{L^2}^2$, which is minimised at $\sigma = \sigma^* (= 0.1)$. 
For each $(N,M)$ we perform $50$ simulations and plot the corresponding $\pl(\sigma)$ as blue curves, which tend to cluster around the red curve as $M$ increases (\Cref{thm:mean_pseudo}). 
The empirical minimisers $\sigma_M^N$ of each simulation are marked in orange, and in the third row we provide histograms showing the empirical distributions of the estimators.

The rows $4$ and $5$ are analogous to the rows $2$ and $3$, with $N = 200$. 
One observes that the empirical loss functions and the estimators are more concentrated than those with $N = 20$, as expected (\Cref{thm:main_detailed,thm:main-estimator}).

In \Cref{fig:cv_plot}, we show the coefficient of variation (CV) of $\sigma_M^N$, i.e.~the empirical standard deviation divided by the empirical mean, computed from $100$ simulations with the parameters $\sigma^* \in \{0.01, 0.05, 0.1, 0.2\}$, $N \in \{10, 20, 40, 80\}$, and $M$ varying from $1$ to $200$.

Being a rescaled standard deviation, the CV decreases at the rate roughly $O(N^{-1/2})$ as expected. 
For $M = 1$ and fixed $N$, all of them start at approximately the same level, and we can also observe that for small $\sigma^*$ and $M$ there is a steep decrease of $\sigma^*$ with $M$.
Intuitively this is due to the fact that in this low density regime the pairing can be correctly recovered with high probability, so increasing $M$ has a similar effect to increasing $N$.
As $M$ and $\sigma^*$ increase this benefit of higher $M$ diminishes.

\subsection{Covariance estimation for bivariate normal distribution}
\begin{figure}
    \centering
    \includegraphics[width=\textwidth]{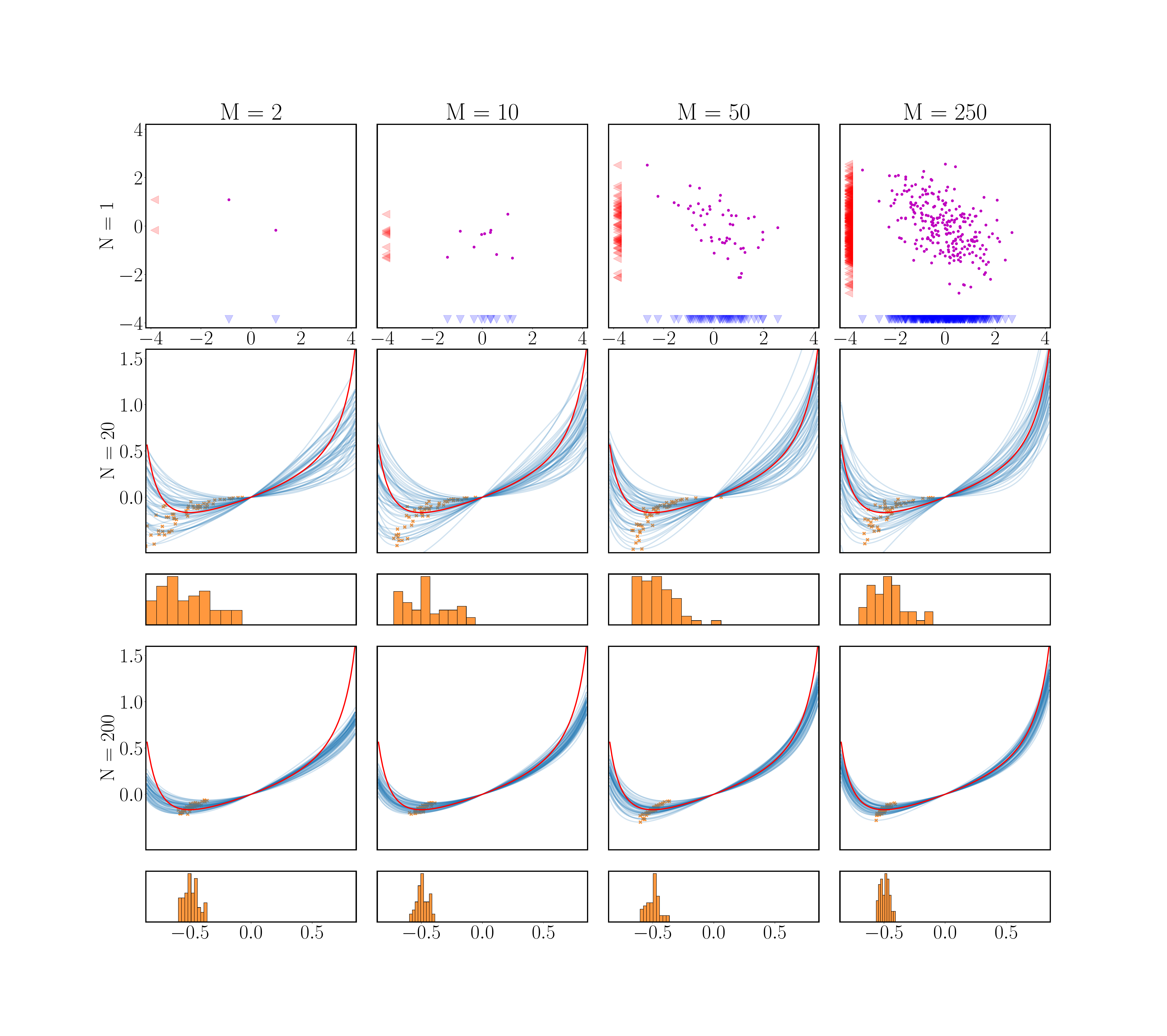}
    \caption{Numerical experiment with broken random samples from a bivariate normal distribution with $\rho^* = -0.5$. The first row shows an example from a single batch; the purple points are “unbroken” $(x,y)$ pairs, whereas in the broken-sample setting only the blue and red marginal points on the axes are observable. The remaining panels are analogous to \Cref{fig:the_plot}, and the red curve is $\lpl(\rho) = \tfrac{1}{2}\lVert p^\rho - p^{\rho^*}\rVert^2_{L^2(\mu\otimes\nu)} + \tfrac{1}{2} - \tfrac{1}{2}\lVert p^{\rho^*}\rVert^2_{L^2(\mu\otimes\nu)}$.}
    \label{fig:bivaGaussian}
\end{figure}
Another interesting setting for the broken sample problem is estimating the correlation coefficient for bivariate normal distribution with known marginals (e.g. see \cite[Section 4]{chan2001file}). Let $\mathbb{X} = \mathbb{Y} = \mathbb{R}$, and
\begin{eqnarray*}
\begin{bmatrix}
X
\\
Y
\end{bmatrix} & \sim & \pi := \mathcal{N}\left(\left[\begin{array}{c}
0\\
0
\end{array}\right],\left[\begin{array}{ccc}
1 & \rho^*\\
\rho^* & 1
\end{array}\right]\right)
\end{eqnarray*}
for some correlation parameter $\rho^* \in (-1,1)$.
Recall that our goal is to estimate $\frac{\diff \pi}{\diff(\mu\otimes\nu)}$, where both $\mu$ and $\nu$ are standard normal distributions on $\mathbb{R}$. Let $\varphi_{X,Y}^\rho$ denote the Lebesgue density of $\pi$ with correlation parameter $\rho$, and let $\varphi_X=\varphi_Y$ denote the Lebesgue density of $\mu=\nu=\mathcal{N}(0,1)$. The parametric density class is then simply taken to be:
\begin{align*}
    p^\rho(x,y):=\frac{\varphi_{X,Y}^\rho(x,y)}{\varphi_X(x)\varphi_Y(y)} \qquad\forall\rho\in[-1+\tau,1-\tau].
\end{align*}
Here, we choose an arbitrary $\tau>0$ to satisfy the conditions in \Cref{def:candidate}. A straightforward calculation yields that, for any $\rho,\rho^*\in[-1+\tau,\,1-\tau]$:
\begin{align*}
    &\norm{p^{\rho^*}}_{L^2(\mu\otimes\nu)}^2 = \frac{1}{(1-{\rho^*}^2)}, 
    &&\norm{p^{\rho^*} - p^\rho}_{L^2(\mu\otimes\nu)}^2
    = \frac{1}{(1-\rho^2)} + \frac{1}{(1-{\rho^*}^2)} - \frac{2}{|\rho\rho^* - 1|}
\end{align*}
We choose the true $\rho^* = -0.5$, and conduct simulations with $M\in\{2,10,50,250\}$ and $N\in\{20,200\}$. In the first row of \Cref{fig:bivaGaussian}, we show a single broken sample batch for varying $M$, note that this only includes the blue and red points on the axis, the purple points are the unbroken points $(x,y)$ therefore would not be available under the broken sample setting. The remaining panels are analogous to \Cref{fig:the_plot} from the torus example and exhibit very similar behaviour of $f_M^N$. It is easy to check that $p^0(x,y) = 1$ for all $(x,y)\in\mathbb{R}^2$, which explains the point shared by the curves at $\rho=0$. 

\bibliographystyle{plain}
\bibliography{bib.bib}
\end{document}